\documentclass[letterpaper,11pt]{article}
\usepackage{multicol}
\usepackage{amsmath,latexsym,amsbsy,amssymb,amsthm}
\usepackage{psfrag,graphicx,epstopdf}
\usepackage{amssymb,latexsym}
\usepackage{mathrsfs, verbatim}
\usepackage[body={15cm, 22cm}]{geometry}

\usepackage{color}
\usepackage{multicol}
\usepackage[pdfencoding=auto, psdextra]{hyperref}
\usepackage{pgfplots}

\usetikzlibrary{arrows.meta, 
                bending,     
                patterns     
               }

\usepackage{enumerate}

\numberwithin{equation}{section}
\newcommand{\norm}[1]{\left\|#1\right\|}
\newcommand{\p}[1]{\left(#1\right)}

\newcommand{\set}[1]{\left\{#1\right\}}
\newcommand{\sett}[2]{\left\{#1 ~:~ #2 \right\}}
\newcommand{\abs}[1]{\left|#1\right|}

\newcommand{\R}{\mathbb{R}}

\newcommand{\mres}{\,\mathbin{\vrule height 1.6ex depth 0pt width
0.13ex\vrule height 0.13ex depth 0pt width 1.3ex}\,}

\def\ds{\displaystyle}

\def\XXint#1#2#3{{\setbox0=\hbox{$#1{#2#3}{\int}$ }
\vcenter{\hbox{$#2#3$ }}\kern-.6\wd0}}

\font\bigbf=cmbx10 at 16pt

\def\ds{\displaystyle}
\def\forall{\mathrm{for~all}~~}
\def\L{{\bf L}}

\def\ve{\varepsilon}
\def\vp{\varphi }

\def\R{\mathbb R}

\def\vp{\varphi}

\def\vs{\vskip 2em}

\def\v{\vskip 1em}

\def\C{{\mathcal C}}

\def\J{{\mathcal J}}

\def\bega{\begin{array}}
\def\enda{\end{array}}
\def\begi{\begin{itemize}}
\def\endi{\end{itemize}}

\def\loc{{\mathrm{loc}}}
\def\bel{\begin{equation}\label}
\def\eeq{\end{equation}}

\def\sqr#1#2{\vbox{\hrule height .#2pt
\hbox{\vrule width .#2pt height #1pt \kern #1pt
\vrule width .#2pt}\hrule height .#2pt }}
\def\square{\sqr74}
\def\endproof{\hphantom{MM}\hfill\llap{$\square$}\goodbreak}

\theoremstyle{plain}
\newtheorem{theorem}{Theorem}[section]
\newtheorem{corollary}[theorem]{Corollary}

\newtheorem{lemma}[theorem]{Lemma}
\newtheorem{proposition}[theorem]{Proposition}

\theoremstyle{definition}
\newtheorem{definition}[theorem]{Definition}

\begin{document}
\title{\bigbf SBV regularity for Burgers-Poisson equation}
\vs
\author{Steven Gilmore and Khai T. Nguyen\\
\\
 {\small  Department of Mathematics, North Carolina State University, }\\
 \\  {\small e-mails: sjgilmo2@ncsu.edu,  ~ khai@math.ncsu.edu}
 }

\maketitle
\begin{abstract}
The SBV regularity of weak entropy solutions to the Burgers-Poisson equation for initial data in $\L^1(\R)$ is considered. We show that the derivative of a solution consists of only the absolutely continuous part and the jump part.
\end{abstract}
\v

\section{General setting}
The Burgers-Poisson equation is given by the balance law obtained from Burgers' equation by adding a nonlocal source term
\begin{equation}\label{BP}
u_t + \left({\frac{u^2}{ 2}}\right)_x=[G*u]_x\,.
\end{equation}
Here, $G(x)=-\frac{1}{ 2}e^{-|x|}$ is the Poisson Kernel such that
\[
[G*f](x)~=~\int_{-\infty}^{+\infty}G(x-y)\cdot f(y)~dy
\]
solves the Poisson equation
\begin{equation}\label{P}
\varphi_{xx}-\varphi~=~f\,.
\end{equation}
Equation (\ref{BP}) has been derived in \cite{W} as a simplified model of shallow water waves and admits conservation of both momentum and energy. For sufficiently regular initial data $u_0$, the local existence and uniqueness of solutions of (\ref{BP}) has been established in \cite{FC}. Additionally, their analysis of traveling waves showed that the equation features wave breaking in finite time. More generally, it has been demonstrated that (\ref{BP}) does not admit a global smooth solution (\cite{H-Liu}). Hence, it is natural to consider entropy weak solutions.

\begin{definition}
A function $u\in \mathbf{L}_{\loc}^{1}([0,\infty[\times\R)\cap {\bf L}^{\infty}_{\loc}(]0,\infty[, {\bf L}^{\infty}(\R))$ is an {\it entropy weak solution} of \eqref{BP} if $u$ satisfies the following properties:
\begin{itemize}
\item [(i)] the map $t\mapsto u(t,\cdot)$ is continuous with values in $\mathbf{L}^1(\R)$, i.e.,
\[
\|u(t,\cdot)-u(s,\cdot)\|_{{\bf L}^1(\R)}~\leq~L\cdot |t-s|\qquad\forall 0\leq s\leq t
\]
for some constant $L>0$.
\item[(ii)] For any $k\in\R$ and any non-negative test function $\phi\in C^1_c(]0,\infty[\times\R,\R)$ one has
\begin{equation*}
\ds\int\int \Big[|u-k|\phi_t+\mathrm{sign}(u-k)\Big(\frac{u^2}{2}-\frac{k^2}{ 2}\Big)\phi_x+\mathrm{sign}(u-k)[G_x*u(t,\cdot)](x)\phi\Big]~dx\,dt~\geq~0\,.
\end{equation*}
\end{itemize}
\end{definition}
 Based on the vanishing viscosity approach, the existence result for a global weak solution was provided for $u_0\in BV(\R)$ in \cite{FC}. However, this approach cannot be applied to the more general case with initial data in ${\bf L}^1(\R)$. Moreover, there are no uniqueness or continuity results for global weak entropy solutions of \eqref{BP} established in \cite{FC}. Recently, the existence and continuity results for global weak entropy solutions of \eqref{BP} were established for ${\bf L}^1(\R)$ initial data in \cite{GN}. The entropy weak solutions are constructed by a flux-splitting method.  Relying on the decay properties of the semigroup generated by Burgers equation and the Lipschitz continuity of solutions to the Poisson equation, approximating solutions satisfy an Oleinik-type inequality for any positive time. As a consequence, the sequence of approximating solutions is precompact and converges in \({\bf L}^1_{\loc}(\R)\). Moreover, using an energy estimate, they show that the characteristics are H\"older continuous, which is used to achieve the continuity property of the solutions. The Oleinik-type inequality gives that the solution $u(t,\cdot)$ is in $BV_{\loc}(\R)$ for every $t>0$. In particular, this implies that the Radon measure $Du(t,\cdot)$ is divided into three mutually singular measures
\[
  Du(t,\cdot)~=~D^au(t,\cdot)+D^ju(t,\cdot)+D^cu(t,\cdot)
\]
where $D^au(t,\cdot)$ is the absolutely continuous measure with respect to the Lebesgue measure, $D^ju(t,\cdot)$ is the {\it jump part} which is a countable sum of weighted Dirac measures, and $D^cu(t,\cdot)$ is the non-atomic singular part of the measure called the {\it Cantor part}. For a given $w\in BV_{\loc}(\R)$, the Cantor part of $Dw$ does not vanish in general. A typical example of $D^cw$ is the derivative of the Cantor-Vitali ternary function. If $D^cw$ vanishes then we say the function $w$ is locally in the space of special functions of bounded variation, denoted by $SBV_{\loc}(\R)$. The space of $SBV_{\loc}$ functions was first introduced in \cite{GL} and plays important role in the theory of image segmentation and with variational problems in fracture mechanics.  Motived by results on $SBV$ regularity for hyperbolic conservation laws (\cite{AC,R, BC,M}), we show that

\begin{theorem}\label{SBV} Let $u:[0,\infty[\times\R\to\R$ be the unique locally $BV$-weak entropy solution of (\ref{BP}) with initial data $u_0\in{\bf L}^1(\R)$. Then there exists a countable set $\mathcal{T}\subset\R^+$ such that
\[
u(t,\cdot)~\in~ SBV_{\loc}(\R)\qquad\forall t\in \R^+\setminus\mathcal{T}\,.
\]
\end{theorem}
As a consequence, the slicing theory of $BV$ functions and the chain rule of Vol'pert \cite{AFP} implies that the weak entropy solution $u$ is in $SBV_{\loc}([0,+\infty[\times\R)$. This is the first example of the SBV regularity for scalar conservation laws with nonlocal source term. A common theme in the proofs of recent results on $SBV$ regularity involve an appropriate geometric functional which has certain monotonicity properties and jumps at time $t$ if $u(t,\cdot)$ does not belong to $SBV$ (see e.g. in \cite{AC}). More precisely, let $\J(t)$ be the set of jump discontinuities $\J(t)$ of $u(t,\cdot)$. For each $x_j \in \J(t)$, there are minimal and maximal backward characteristics $\xi_j^-(s)$ and $\xi_j^+(s)$ emanating from $(t,x_j)$ which define a nonempty interval $\ds I_j(s) := \,]\xi_j^-(s),\xi_j^+(s)[$ for any $s <t$. In this case, the functional $F_s(t)$ defined as the sum of the measures of $I_j(s)$ is monotonic and bounded. Relying on a careful study of generalized characteristics, one shows that if the measure $Du(t,\cdot)$ has a non-vanishing Cantor part then the function $F_s$ ``jumps'' up at time $t$ which implies that the Cantor part is only present at countably many $t$. Due to the nonlocal source,  $u(t,\cdot)$ does not necessarily have compact support. Thus, we approach the domain by first looking  at compact sets and then ``glue'' the sections together to recover the full domain.

%


\section{Preliminaries}
\subsection{{\it BV} and {\it SBV} functions}
Let us now introduce the concept of functions of bounded variation in $\R$. We refer to \cite{AFP} for a comprehensive analysis.
\begin{definition} Given an open set $\Omega\subseteq\R$, let $w$ be in $\L^1(\Omega)$. We say that $w$ is a {\it function of bounded variation in} $\Omega$ (denoted by $w \in BV(\Omega)$) if the distributional derivative of $w$ is representable by a finite Radon measure $Du$ on $\Omega$, i.e.,
    \[
      - \int_\Omega w \cdot \varphi' ~ dx ~=~ \int_\Omega \vp ~dDw \qquad \forall \vp \in \C^\infty_c(\Omega)
    \]
    with {\it total variation} (denoted by $\norm{Dw}$) given by
    \[
      \norm{Dw}(\Omega) ~=~ \sup \sett{ \int_\Omega w \cdot \varphi'~ ~ dx }{ \vp \in \C^\infty_c(\Omega),\,\, \|\vp\|_{{\bf L}^{\infty}}\leq 1} \,.
    \]
    Moreover, $w$ is of {\it locally bounded variation} on $\Omega$ (denoted by $w \in BV_{\loc}(\Omega)$) if $w \in {\bf L}^1_{\loc}(\Omega)$ and $w$ is in $BV(U)$ for all $U \subset\subset \Omega$.
\end{definition}
Given $w\in BV_{loc}(\R)$, we split $Dw$ into the {\it absolutely continuous part} $D^a w$ and {\it singular part} $D^s w$ provided by the Radon-Nikod\'ym theorem (see e.g. \cite[Theorem 1.28]{AFP}). In the $1$-D case, the singular part is concentrated on the ${\bf L}^1$-negligible set
\[
S_w~=~\left\{t\in\R~\Big|~\lim_{\delta\to 0}\dfrac{|Dw|(t-\delta,t+\delta)}{ |\delta|}~=~+\infty\right\}\,.
\]
We can further decompse $D^sw$ by isolating the set of atoms
$A_w=\left\{t\in\R~\big|~Dw(\{t\})\neq0\right\}$,
contained in $S_w$. Hence, we can consider two mutually singular measures
\[
    D^jw~:=~ D^sw \mres A_w\qquad\mbox{and} \qquad D^cw~:=~D^sw \mres (S_w \setminus A_w)
\]
respectively called the {\it jump part} of the derivative and the {\it Cantor part} of the derivative. Furthermore, we have the following structure result (see e.g. \cite[Theorem 3.28]{AFP})
\begin{proposition} Let $\Omega \subseteq \R$ and $w \in BV(\Omega)$. Then, for any $x \in A_w$, the left and right hand limits of $w(x)$ exist and
  \[
      D^jw ~=~ \sum_{x\in A_w} \left(w(x+) - w(x-)\right) \delta_x
  \]
where $w(x\pm)$ denote the one-sided limits of $w$ at $x$. Moreover, $D^cw$ vanishes on any sets which are $\sigma$-finite with respect to $\mathcal{H}^0$.
\end{proposition}
\begin{definition} Let $w$ be in $BV_{loc}(\R)$ then $w$ is a {\it special function of bounded variation} (denote by $w\in SBV$) if the Cantor part $D^cw$ vanishes.
\end{definition}
We want to show that the weak entropy solutions of \eqref{BP} belong to $SBV$.

\subsection{Oleinik-type inequality and non-crossing of characteristics }
The global existence and $BV$-regularity of \eqref{BP} was studied extensively in \cite{GN}. For convenience, we recall their main results here.
\begin{theorem}  The Cauchy problem \eqref{BP}-\eqref{P} with initial data $u_0 = u(0,\cdot) \in {\bf L}^1(\R)$ admits a unique solution $u(t,x)$ such that for all $t>0$ the following hold:
  \begin{enumerate}[(i)]
    \item the ${\bf L}^1$-norm is bounded by
    \begin{equation}\label{L1bound}
    \|u(t,\cdot)\|_{{\bf L}^{1}(\R)}~\leq~e^{t}\cdot \|u_0\|_{{\bf L}^1(\R)} \,;
    \end{equation}
    \item the solution satisfies the following Oleinik-type inequality
    \begin{equation}\label{O}
      u(t,y) - u(t,x) ~\leq~ \frac{K_t}{t}\cdot(y -x) \qquad \forall~ y > x
    \end{equation}
    with $K_t ~=~ 1+2t+2t^2+4t^2e^{t}\cdot \|u_0\|_{{\bf L}^1(\R)}$;
    \item the ${\bf L}^\infty$-norm is bounded by
    \begin{equation}\label{LinftyB}
      \norm{u(t,\cdot)}_{{\bf L}^\infty(\R)}~\leq~\sqrt{ \frac{2 K_t}{t} \norm{u(t,\cdot)}_{{\bf L}^1(\R)}}~\leq~\sqrt{\frac{2K_te^t}{t} \norm{u_0}_{{\bf L}^1(\R)}} \,.
   \end{equation}
 \end{enumerate}
\end{theorem}
In particular, this implies that that for all $t>0$, $u(t,\cdot)$ is in $BV_{\loc}(\R)$ and satisfies 
\begin{equation}\label{Lax}
  u(t,x-)~\geq~u(t,x+) \qquad \forall x \in \R.
\end{equation}
We recall the definition and theory of generalized characteristic curves associated to (\ref{BP}). For a more in depth theory of generalized characteristics, we direct the readers to \cite{CD1}.

\begin{definition}
For any $(t,x)\in ]0,+\infty[\times\R$, an absolutely continuous curve $\xi_{(t,x)}(\cdot)$ is called a {\it backward characteristic curve} starting from $(t,x)$ if it is a solution of differential inclusion
\begin{equation}\label{d-incl-b}
\dot \xi_{(t,x)}(s)~\in~\left[u\left(s,\xi_{(x,t)}(s)+\right),u\left(s,\xi_{(t,x)}(s)-\right)\right]\qquad a.e.~s\in [0,t]
\end{equation}
with $\xi_{(t,x)}(t)=x$. If $s\in [t,+\infty[$ in (\ref{d-incl-b}) then $\xi$ is called a {\it forward characteristic curve}, denoted by $\xi^{(t,x)}(\cdot)$. The characteristic curve $\xi$ is called {\it genuine} if $u(t,\xi(t)-)=u(t,\xi(t)+)$ for almost every $t$.
\end{definition}

The existence of backward (forward) characteristics was studied by Fillipov.  As in \cite{CD1} and \cite{R}, the speed of the characteristic curves are determined and genuine characteristics are essentially classical characteristics:

\begin{proposition}\label{Charact} Let $\xi: [a,b]\to\R$ be a characteristic curve for the Burgers-Poisson equation (\ref{BP}), associated with an entropy solution $u$. Then for almost every time $t\in [a,b]$, it holds that
\begin{equation}\label{PCharac}
\dot{\xi}(t)~=~
\begin{cases}
u(t,\xi(t))\qquad&\mathrm{if}\qquad  u(t,\xi(t)+)~=~u(t,\xi(t)-)  \,,\\[3mm]
\dfrac{u(t,\xi(t)+)+u(t,\xi(t)-)}{ 2 }\qquad  &\mathrm{if}\qquad  u(t,\xi(t)+)~<~u(t,\xi(t)-)\,.
\end{cases}
\end{equation}
In addition, if $\xi$ is genuine on $[a,b]$, then there exists $v(t) \in C^1([a,b])$ such that
\[
  u(t,\xi(t)-) ~=~ v(t) ~=~ u(t,\xi(t)+) \qquad \forall t \in ]a,b[
\]
and $(\xi(\cdot),v(\cdot))$ solve the system of ODEs
\begin{equation}\label{ODE-C}
\begin{cases}
  \dot \xi(t)~=~ v(t) \\[2mm]
  \dot v(t)~=~ [G*u(t,\cdot)]_x(\xi(t))
\end{cases}
\qquad\forall t\in ]a,b[ \,.
\end{equation}
\end{proposition}
Backward characteristics $\xi_{(t,x)}(\cdot)$ are confined between a maximal and minimal backward characteristics, as defined in \cite{CD1} $\big($denoted by $\xi_{(t,x+)}(\cdot)$ and $\xi_{(t,x-)}(\cdot)$ $\big)$. Relying on the above proposition and (\ref{Lax}), we can obtain properties of generalized characteristics, associated with entropy solutions of the Burgers-Poisson equation, including the non-crossing property of two genuine characteristics.

\begin{proposition}\label{gen-char} Let $u$ be an entropy solution to \eqref{BP}. Then for any $(t,x)\in ]0+\infty[\times\R$, the following holds:
\begin{itemize}
\item [(i)] The maximal and minimal backward characteristics $\xi_{(t,x\pm)}$ are genuine and thus the function $u\left(\tau,\xi_{(t,x\pm)}(\tau)\right)$ solves (\ref{ODE-C}) for $\tau \in ]0,t[$ with initial data $u(t,\xi_{(t,x\pm)}(t))$.
\item [(ii)] [Non-crossing of genuine characteristics] Two genuine characteristics may intersect only at their endpoints.
\item [(iii)] If $u(t,\cdot)$ is discontinuous at a point $x$, then there is a unique forward characteristic $\xi^{(t,x)}$ which passes though $(t,x)$ and
\[
u\left(\tau,\xi^{(t,x)}(\tau)-\right)~>~u\left(\tau,\xi^{(t,x)}(\tau)+\right)\qquad\forall \tau \geq t \,.
\]
\end{itemize}
\end{proposition}
Throughout this paper, we shall denote by $\mathcal{J}(t)=\{x\in\R: u(t,x-)>u(t,x+)\}$, the jump set of $u(t,\cdot)$ for any $t>0$. For any $x\in \mathcal{J}(t)$, the base of the backward characteristic cone starting from $(t,x)$ at time $s \in [0,t[$ is
\begin{equation}\label{I}
I_{(t,x)}(s)~:=~]\xi_{(t,x-)}(s),\xi_{(t,x+)}(s)[.
\end{equation}
By the non-crossing property, for any $T>0$ and $z_1<z_2\in \R\setminus \mathcal{J}(T)$, the set
\begin{equation}\label{A}
\mathcal{A}^T_{[z_1,z_2]}~:=~\bigcup_{s\in [0,T]}A_{[z_1,z_2]}^T(s)\qquad\mathrm{with}\qquad A_{[z_1,z_2]}^T(s)~:=~]\xi_{(T,z_1)}(s),\xi_{(T,z_2)}(s)[
\end{equation}
confines all backward characteristics starting from $(T,x)$ with $x\in ]z_1,z_2[$. For any $0<s<\tau\leq T$, we denote by
\begin{equation}\label{Its}
I^{\tau,T}_{[z_1,z_2]}(s)~=~\bigcup_{x\in A_{[z_1,z_2]}^T(\tau)\bigcap \mathcal{J}(\tau)}I_{(\tau,x)}(s).
\end{equation}
Due to the no-crossing property of two genuine backward characteristics and the uniqueness of  forward characteristics in Proposition \ref{gen-char}, the following holds:
\begin{corollary}\label{base-inc}
Given $T>0$ and $z_1<z_2\in \R\setminus \mathcal{J}(T)$, the map $\tau\mapsto I_{[z_1,z_2]}^{\tau,T}(s)$ is increasing in the interval $]s,T]$ in the following sense
\begin{equation}\label{I-inc}
I_{[z_1,z_2]}^{\tau_1,T}(s)~\subseteq~I_{[z_1,z_2]}^{\tau_2,T}(s)\qquad\forall 0 \leq s < \tau_1\leq \tau_2 \leq T.
\end{equation}
Moreover, for any $x\in A^T_{[z_1,z_2]}(\tau_1) \setminus  I^{\tau_2,T}_{[z_1,z_2]}(\tau_1)$ with $0<\tau_1<\tau_2<t$, the unique forward characteristic $\xi^{(\tau_1,x)}$ passing through $(\tau_1,x)$ is genuine in $[\tau_1,\tau_2]$.
\end{corollary}
\begin{proof} Let $x \in \J(\tau_1) \cap  A^T_{[z_1,z_2]}(\tau_1) $ and let $\chi(\cdot)$ be the unique forward characteristic emenating from $(\tau_1,x)$.
By property (iii) of Proposition \ref{gen-char}, for a fixed $\tau_2 \in [\tau_1,T]$ we have that $\chi(\tau_2) \in \J(\tau_2)$ and by the non-crossing property, $\chi(\tau_2) \in A^T_{[z_1,z_2]}(\tau_2)$.
Since the backward characteristics that form the base of a characteristic cone are genuine, the non-crossing property implies that
\[
  I_{(\tau_1,x)}(s) \subseteq I_{(\tau_2,\chi(\tau_2))}(s) \subset A^T_{[z_1,z_2]}(s) \qquad \forall s \in [0,\tau_1]
\] yielding \eqref{I-inc}. The later statement follows directly.
\end{proof}

\section{SBV-regularity} Throughout this section, let  $u:[0,\infty[\times\R\to\R$ be the unique locally $BV$-weak entropy solution of (\ref{BP}) for some initial data $u_0\in{\bf L}^1(\R)$. The section aims to  prove Theorem \ref{SBV}. For simplicity, denote the jump and Cantor parts of $Du(t,\cdot)$ by
\[
  \nu_t~=~D^ju(t,\cdot)\qquad\mathrm{and}\qquad \mu_t~=~D^cu(t,\cdot) \qquad \mbox{for any}~~ t \in ]0,+\infty[\,
\]
which, by \eqref{O}, are both non-positive. We will show that $\mu_t(\R)<0$  for at most countable positive times $t>0$. In order to do so,  let us first establish some basic bounds on backward characteristics.

\begin{lemma}\label{bsc} For any given $0<t_0<t$ and $x_1\leq x_2$, let $\xi_i(\cdot)$ be a genuine backward characteristic starting from $(t,x_i)$ and
\[
v_i(s)~=~u(s,\xi_i(s))\quad\forall s\in [0,t],~ i\in \{1,2\}.
\]
Then the followings hold:
\begin{align}\label{eet1}
|v_2(s)-v_1(s)|+|\xi_2(s)-\xi_1(s)|~\leq~c_t(s) \cdot\p{|v_2(t)-v_1(t)|+|\xi_2(t)-\xi_1(t)|}
\end{align}
for all $s\in [0,t]$ and
\begin{equation}\label{eet2}
\xi_2(t_0)-\xi_1(t_0)~\geq~ \frac{x_2-x_1+(v_1(t_0)-v_2(t_0))\cdot (t-t_0)}{ \Gamma_{[t_0,t]}}
\end{equation}
with
\begin{equation}\label{cs}
\begin{cases}
c_t(s)&=~ \exp \set{ 2\cdot\left(\sqrt{ 2 K_te^t \norm{u_0}_{{\bf L}^1(\R)}} +(e^t\norm{u_0}_{{\bf L}^1(\R)}+1)\cdot \sqrt{t}\right)\cdot (\sqrt{t}-\sqrt{s})}  ,\cr\cr
\Gamma_{[t_0,t]}&=~1+\left(\sqrt{\dfrac{2K_t e^t}{t_0} \norm{u_0}_{{\bf L}^1(\R)}}+e^t\norm{u_0}_{{\bf L}^1(\R)}\right)\ds\cdot \dfrac{e^{K_t}t }{ t_0}\cdot (t-t_0)^2.
\end{cases}
\end{equation}
\end{lemma}
\begin{proof} {\bf 1.} Let's first proof (\ref{eet1}). From Proposition \ref{Charact}, it holds that
\begin{equation}\label{char1}
\begin{cases}
\dot{\xi}_i(s)&=~v_i(s)\\[3mm]
\dot{v}_i(s)&=~[G*u(s,\cdot)]_x(\xi_i(s))
\end{cases}
\quad\forall s\in ]0,t[.
\end{equation}
In particular, this implies that
\[
\frac{d}{ds}~\big|\xi_2(s)-\xi_1(s)\big|~\geq~-\left|v_2(s)-v_1(s)\right|
\]
and
\[
\frac{d}{ds}~\big|v_2(s)-v_1(s)\big|~\geq~-\Big|[G*u(s,\cdot)]_x(\xi_2(s))-[G*u(s,\cdot)]_x(\xi_1(s))\Big|.
\]
Since $\xi_2(s)\geq \xi_1(s)$ for all $s\in ]0,t]$,  we estimate
\begin{align*}
 &\Big|[G*u(s,\cdot)]_x(\xi_2(s))-[G*u(s,\cdot)]_x(\xi_1(s))\Big|
~\leq~\frac{1}{2}\cdot \int_{-\infty}^{\xi_1(s)} \left| u(s,z)\right| \cdot \left| e^{z - \xi_2(s)} - e^{z-\xi_1(s)}\right| ~dz \\
 &+ \frac{1}{2}\cdot \int_{\xi_1(s)}^{\xi_2(s)}\left|  u(s,z)\right| \cdot \left| e^{z - \xi_2(s)} + e^{\xi_1(s) -z}\right| ~dz
+ \frac{1}{2} \cdot\int_{\xi_2(s)}^{+\infty}\left|u(s,z)\right| \cdot \left| e^{\xi_1(s)- z}-
e^{\xi_2(s) -z}\right| ~dz\\
&\qquad~\leq~\frac{1}{2} \cdot \left(1-e^{\xi_1(s)-\xi_2(s)}\right)\int_{\R\setminus [\xi_1(s),\xi_2(s)] } |u(s,z)|~dz+\int_{ [\xi_1(s),\xi_2(s)]} |u(s,z)|~dz\\
&\qquad~\leq~\left(\frac{1}{ 2}\cdot \|u(s,\cdot)\|_{{\bf L}^{1}(\R)}+\|u(s,\cdot)\|_{{\bf L}^{\infty}(\R)}\right)\cdot \big|\xi_2(s)-\xi_1(s)\big| \,.
\end{align*}
Hence,  (\ref{L1bound}) and (\ref{LinftyB}) imply that
\begin{multline}\label{G-lip}
\Big|[G*u(s,\cdot)]_x(\xi_2(s))-[G*u(s,\cdot)]_x(\xi_1(s))\Big|\\~\leq~\left(\sqrt{\frac{2K_te^t}{s} \norm{u_0}_{{\bf L}^1(\R)}}+e^t\norm{u_0}_{{\bf L}^1(\R)}\right)\cdot \big|\xi_2(s)-\xi_1(s)\big|.
\end{multline}
Setting $M_t=\sqrt{2 K_te^t \norm{u_0}_{{\bf L}^1(\R)}}+(e^t\norm{u_0}_{{\bf L}^1(\R)}+1)\cdot \sqrt{t}$, we have
\[
\frac{d}{ds}\big(\big|\xi_2(s)-\xi_1(s)\big|+\big|v_2(s)-v_1(s)\big|\big)~\geq~-\frac{M_t}{\sqrt{s}}\cdot \big(\big|\xi_2(s)-\xi_1(s)\big|+\big|v_2(s)-v_1(s)\big|\big),
\]
for all $s \in ]0,t]$, and Gr\"onwall's inequality yields (\ref{eet1}).
\v

{\bf 2.} To prove (\ref{eet2}), we first apply (\ref{O}) to \eqref{char1} to get
\[
\dot{\xi}_2(s)-\dot{\xi}_1(s)~=~u(s,\xi_2(s))-u(s,\xi_1(s))~\leq~\frac{K_{t}}{s}\cdot (\xi_2(s)-\xi_1(s)),
\]
and this implies
\begin{equation}\label{O-char}
\xi_2(s)-\xi_1(s)~\leq~\frac{e^{K_t}s}{t_0}\cdot (\xi_2(t_0)-\xi_1(t_0))~\leq~\frac{e^{K_t}t}{t_0}\cdot (\xi_2(t_0)-\xi_1(t_0))\quad\forall s\in [t_0,t].
\end{equation}
Therefore, from (\ref{char1}) and \eqref{G-lip}, it holds for $s\in [t_0,t]$ that
\begin{align*}
v_2(s)-v_1(s)~=&~v_2(t_0)-v_1(t_0)+\int_{t_0}^s[G*u(\tau,\cdot)]_x(\xi_2(\tau))-[G*u(\tau,\cdot)]_x(\xi_1(\tau))~d\tau  \\
~\leq&~v_2(t_0)-v_1(t_0)+\int_{t_0}^{s}\left(\sqrt{\frac{2K_te^t}{t_0} \norm{u_0}_{{\bf L}^1(\R)}}+e^t\norm{u_0}_{{\bf L}^1(\R)}\right)\cdot \big(\xi_2(\tau)-\xi_1(\tau)\big)~d\tau \\
~\leq&~v_2(t_0)-v_1(t_0)+\gamma_{[t_0,t]}\cdot (\xi_2(t_0)-\xi_1(t_0))
\end{align*}
with
\[
\gamma_{[t_0,t]}~=~\left(\sqrt{\frac{2K_t e^t}{t_0} \norm{u_0}_{{\bf L}^1(\R)}}+e^t\norm{u_0}_{{\bf L}^1(\R)}\right)\cdot\frac{e^{K_t}t}{t_0}\cdot (t-t_0).
\]
Integrating the first equation in \eqref{char1} over $[t_0,t]$,  we get
\begin{align*}
 \xi_2(t)-\xi_1(t)~=&~ \xi_2(t_0)-\xi_1(t_0)+\int_{t_0}^{t}v_2(\tau)-v_1(\tau)~d\tau\\
 ~\leq&~\left(v_2(t_0)-v_1(t_0)\right)\cdot (t-t_0)+\left(1+\gamma_{[t_0,t]}\cdot (t-t_0)\right)\cdot\left(\xi_2(t_0)-\xi_1(t_0)\right)
 \end{align*}
 and this yields (\ref{eet2}).
\end{proof}
\v

As a consequence, we obtain the following two corollaries. The first one provides an upper bound on the base of characteristic cone $C_{(t,x)}$ at time $s\in ]0,t[$ for every $x\in \mathcal{J}(t)$.
\medskip

\begin{corollary}\label{base-bound} For any $(t,x)\in ]0,+\infty[\times\mathcal{J}(t)$, it holds that
\begin{equation}\label{Low-length}
\left|I_{(t,x)}(s)\right|~\leq~-c_t(s)\cdot \nu_t(\set{x})\quad\forall s\in [0,t[.
\end{equation}
\end{corollary}
\begin{proof} Since $x\in \mathcal{J}(t)$, the inequality (\ref{Lax}) implies that
\[
\nu_t\{x\}~=~u(t,x+)-u(t,x-)~<~0.
\]
Thus, recalling (\ref{eet1}), we obtain
\begin{eqnarray*}
\big|\xi_{(t,x+)}(s)-\xi_{(t,x-)}(s)\big|&\leq&c_t(s)\cdot \left|u(t,x+)-u(t,x-)\right|
\end{eqnarray*}
and this yields (\ref{Low-length}).
\end{proof}
\v
\medskip

In the next corollary, we show that two distinct characteristics are separated for all positive time; moreover, the distance between them is proportional to the difference in the values of the solution along the characteristics.
\medskip

\begin{corollary}\label{rarefaction} Given $x_1 < x_2$ and $\sigma$ and $t$, such that $0 < \sigma < t \leq T$, let $\xi_i(\cdot)$ be a genuine backward characteristic starting from $(t,x_i)$ and
\[
  v_i(s)~=~u(s,\xi_i(s))\quad\forall s\in [0,t[,~ i\in \{1,2\}.
\]
Then it holds that
\begin{equation}\label{eet3}
  \xi_2(\sigma/2) - \xi_1(\sigma/2) ~\geq~ \kappa_{[\sigma,T]} \cdot \left( v_1(t) - v_2(t) \right) 
\end{equation}
where
\[
  \kappa_{[\sigma,T]} ~=~\dfrac{\sigma}{2}\left[ \Gamma_{[\sigma/2,T]} +\left(\sqrt{\frac{4 K_T e^T}{\sigma} \norm{u_0}_{{\bf L}^1(\R)}}+e^T\norm{u_0}_{{\bf L}^1(\R)}\right)\cdot e^{K_T}T\cdot (T-\sigma/2)  \right]^{-1} \,.
\]
\end{corollary}
\begin{proof}
Integrating the second equation in \eqref{ODE-C} over $[\sigma/2, t]$ yields
\begin{align*}
v_1(t)-v_2(t) ~=&~ v_1(\sigma/2)-v_2(\sigma/2)
+ \int_{\sigma/2}^t[G*u(\tau,\cdot)]_x(\xi_1(\tau))-[G*u(\tau,\cdot)]_x(\xi_2(\tau))~d\tau  \\
~\leq&~ v_1(\sigma/2)-v_2(\sigma/2)
+ \int_{\sigma/2}^t \Big| [G*u(\tau,\cdot)]_x(\xi_2(\tau))-[G*u(\tau,\cdot)]_x(\xi_1(\tau)) \Big| ~d\tau
\end{align*}
and by \eqref{G-lip} and \eqref{O-char} it holds that
\begin{equation}\label{est-v}
\begin{aligned}[c]
v_1(t)&-v_2(t) ~\leq~ v_1(\sigma/2)-v_2(\sigma/2) \\
&+ \left(\sqrt{\frac{4K_T e^T}{\sigma} \norm{u_0}_{{\bf L}^1(\R)}}+e^T\norm{u_0}_{{\bf L}^1(\R)}\right)\cdot\frac{2e^{K_T}T}{\sigma}\cdot (T-\sigma/2) \cdot (\xi_2(\sigma/2) - \xi_1(\sigma/2) ) \,.
\end{aligned}
\end{equation}
On the other hand, by \eqref{eet2} we have that
\[
  v_1(\sigma/2) - v_2(\sigma/2) ~\leq~ \dfrac{\Gamma_{[\sigma/2,t]}}{t-\sigma/2}\cdot (\xi_2(\sigma/2)- \xi_1(\sigma/2)) ~\leq~ \dfrac{2\Gamma_{[\sigma/2,T]}}{\sigma}\cdot (\xi_2(\sigma/2)- \xi_1(\sigma/2)).
\]
which, when applied to \eqref{est-v}, implies \eqref{eet3}.
\end{proof}
\v

 The next lemma shows that, for a certain positive time $s$, if $u(s,\cdot)$ is not in $SBV$, then at future times $s+ \ve$ the Cantor part of $u(s, \cdot)$ gets transformed into jump singularities. Following the main idea in \cite{AC,R}, for any $s \in ]0,T[$ and $z_1 < z_2 \in \R \setminus \J(T)$, let us consider the set of points $E^T_{[z_1,z_2]}(s)$ in $A^T_{[z_1,z_2]}(s)$ where the Cantor part of $D_xu(s,\cdot)$ prevails, i.e.,
\begin{equation}\label{E}
 E^T_{[z_1,z_2]}(s)~=~\left\{x\in A^T_{[z_1,z_2]}(s):\lim_{\eta\to 0+}\frac{\eta+|D_xu(s,\cdot)-\mu_{s}|([x-\eta,x+\eta])}{ -\mu_{s}([x-\eta,x+\eta])}~=~0\right\}.
\end{equation}
Besicovitch differentiation theorem \cite{AFP} gives that $\mu_{s}\left(A^T_{[z_1,z_2]}(s)\setminus E^T_{[z_1,z_2]}(s)\right)=0$ and
\begin{equation}\label{c1}
 \lim_{\eta \to 0^+} \frac{u^-(s,x-\eta) - u^+(s, x+\eta)}{-\mu_{s}([x-\eta,x+\eta])} ~=~ 1\qquad\forall x\in E^T_{[z_1,z_2]}(s).
\end{equation}
Moreover, for $\mu_{s}$-a.e. $x$ in $E^T_{[z_1,z_2]}(s)$, it holds that
\begin{equation}\label{cantor}
\lim_{\eta\to 0} \frac{u(s,x+\eta)-u(s,x)}{\eta}~=~-\infty.
\end{equation}

\begin{lemma}\label{cantor-jump} Let $0<s<t\leq T$ and $z_1 < z_2 \in \R \setminus \J(T)$ be fixed. Then, it holds for $\mu_{s}$-a.e.  $x\in A^T_{[z_1,z_2]}(s)$ that
\[
      ]x-\eta_x, x+\eta_x [  ~\subset~ I^{t,T}_{[z_1,z_2]}(s)\qquad\mathrm{for~some~}\eta_x>0.
\]
\end{lemma}
\begin{proof}
Since $I^{t,T}_{[z_1,z_2]}(s)$ is open, it is sufficient to prove that every point $x \in E^T_{[z_1,z_2]}(s)\setminus \J(s)$ satisfying (\ref{cantor}) is in  $I^{t,T}_{[z_1,z_2]}(s)$.  Assume by a contradiction that
\[
x~\in~A^T_{[z_1,z_2]}(s)\setminus \overline{ I^{t,T}_{[z_1,z_2]}(s)}\bigcup  \partial (\overline{ I^{t,T}_{[z_1,z_2]}(s) }).
\]
{\bf 1.} If $x\in A^T_{[z_1,z_2]}(s)\setminus \overline{ I^{t,T}_{[z_1,z_2]}(s)}$ then
    \begin{equation}\label{cdc}
      ]x-\eta_0, x + \eta_0[~ \bigcap ~\overline{ I^{t,T}_{[z_1,z_2]}(s) }~=~ \emptyset \qquad\mathrm{for~some}~\eta_0>0.
    \end{equation}
Given any $\eta\in [0,\eta_0[$, let  $\xi^{\eta}_1(\cdot)$ and $\xi^{\eta}_2(\cdot)$ be the unique forward characteristics emanating from $x-\eta$ and $x+\eta$ at time $\tau_0$.
From Corollary \ref{base-inc}, both $\xi_1^{\eta}(\cdot)$ and $\xi_2^{\eta}(\cdot)$ are genuine in $[t_0,t]$ and
\begin{equation}\label{nocr}
\xi_2^{\eta}(\tau)-\xi_1^{\eta}(\tau)~\geq~0\qquad\forall \tau \in [s,t]
\end{equation}
Thus, (\ref{eet2}) in Lemma \ref{bsc} implies
\begin{align*}
2\eta~=&~\xi_2^{\eta}(s)-\xi_1^{\eta}(s)~\geq~
\frac{\xi_2^{\eta}(t)-\xi_1^{\eta}(t)+\left(u(s,x-\eta)-u(s,x+\eta)\right)\cdot (t-s)}{ \Gamma_{[s,t]}}\\
~\geq&~-\frac{ \left(u(s,x+\eta)-u(s,x-\eta)\right)\cdot (t-s)}{ \Gamma_{[s,t]}}
\end{align*}
which yields a contradiction to (\ref{cantor}) when $\eta$ is sufficiently small.
\medskip

{\bf 2.}  Suppose that $x \in \partial (\overline{ I^{t,T}_{[z_1,z_2]}(s)})$. In this case, $\xi_{(s,x)}(\cdot)$ is either a minimal or maximal backward characteristic in $[s,t]$.
Moreover, for every $\eta>0$ there exists $x_{\eta}\in ]x-\eta,x[\bigcup ]x,x+\eta[ $ such that $x_\eta\notin \overline{ I^{t,T}_{[z_1,z_2]}(s)}$ and the unique forward characteristics $\xi^{(s,x_\eta)}(\cdot)$ emenating from $x_{\eta}$ at time $s$ is genuine  and does not cross $\xi_{(s,x)}(\cdot)$ in the time interval $[s,t]$. With the same computation in the previous step, we get
\[
\frac{u(s,x_\eta)-u(s,x)}{x_\eta-x}~\geq~-\frac{\Gamma_{[s,t]}}{t-s}
\]
and this also yields a contradiction to (\ref{cantor}) when $\eta$ is sufficiently small.
\end{proof}

We are now ready to prove our first main theorem.
\v

 {\bf Proof of Theorem \ref{SBV}}. The proof is divided into two steps:
\medskip

 {\bf Step 1}.
Fix $T >0$ and $z_1,z_2\in \R\setminus \mathcal{J}(T)$ with $z_1<z_2$ and, recalling (\ref{A}) and (\ref{Its}) let
 \begin{equation}\label{A2}
 \mathcal{A}~=~\mathcal{A}^T_{[z_1,z_2]},  \qquad A_t~=A^T_{[z_1,z_2]}(t)\qquad \mathrm{and}\qquad I^t(s) ~=~ I^{t,T}_{[z_1,z_2]}(s)
 \end{equation}
 for all $0<s<t\leq T$. We claim that the set
\begin{equation}\label{can1}
\mathcal{T}_{[z_1,z_2]}~:=~\left\{t\in [0,T]~:~\mu_t\left(A_t \right) ~\mathrm{does~not~vanish}\right\}
\end{equation}
is at most countable.
\v

{\bf (i).} Fix $\sigma\in ]0,T[$. By Proposition \ref{Charact} and (\ref{LinftyB}), one has
\[
\left|A_t\right|~\leq~|z_2-z_1|+ 2 \sqrt{\frac{2K_{T}e^T}{\sigma} \norm{u_0}_{{\bf L}^1(\R)}}\cdot T\qquad\forall t\in [\sigma,T],
\]
and the Oleinik-type inequality (\ref{O}) yields
\[
|Du(t,\cdot)|\big(A_t\big)~\leq~M_{\sigma}^T\qquad\forall t\in [\sigma,T]
\]
with
\[
M_{\sigma}^T~=~2\sqrt{\frac{2K_{T}e^T}{\sigma} \norm{u_0}_{{\bf L}^1(\R)}}+ \frac{2 K_T}{\sigma}\cdot \left(|z_2-z_1|+2\sqrt{\frac{2K_{T}e^T}{\sigma} \norm{u_0}_{{\bf L}^1(\R)}}\cdot T\right).
\]
Let the geometric functional $F_{\sigma}: [\sigma,T]\to [0,\infty[$  be defined by
 \[
    F_\sigma(t) ~=~ \left|\bigcup_{x \in \J(t)\bigcap A_t}I_{(t,x)}\p{\sigma/2}\right| ~=~\sum_{x \in \J(t)\bigcap A_t} \abs{ I_{(t,x)}\p{\sigma/2}} \qquad \forall t \in [\sigma,T]
  \]
where the second equality follows by the non-crossing property. By Corollaries \ref{base-inc} and \ref{base-bound}, the map $t\mapsto F_{\sigma}(t)$ is non-decreasing in $[\sigma,T]$ and uniformly bounded
\begin{equation}\label{F-bound}
\sup_{t\in \left[\sigma,T \right]} F_{\sigma}(t)~\leq~c_{T}(\sigma/2)\cdot\sup_{t\in\sigma,T}\left(|\nu_t|(A_t)\right)~\leq~c_{T}(\sigma/2)\cdot M_{\sigma}^T
\end{equation}
with $c_T(\sigma/2)$ defined in (\ref{cs}).
\v
\noindent {\bf(ii).} Assume that a Cantor part is present in  $\mathcal{A}$ at  time $t\in ]\sigma, T[$, i.e.,
\begin{equation}\label{cantor-present}
\mu_t\left(A_t\right)~\leq~-\alpha\qquad\mathrm{for~some~} \alpha>0,
\end{equation}
which by \eqref{E} is concentrated on $E_t := E^T_{[z_1,z_2]}(t)$. We will show that
\begin{equation}\label{f-jump}
  F_{\sigma}(t+)-F_\sigma(t)~\geq~\frac{\kappa_{[\sigma,T]}}{2}\cdot \alpha
\end{equation}
where $\kappa_{[\sigma,T]}$ is defined in Corollary \ref{rarefaction}.
It is sufficient to prove that
\[
 F_\sigma(t + \ve) - F_\sigma(t)~=~\left|  I^{t+\ve}(\sigma/2)\setminus I^t(\sigma/2)   \right|~\geq~\frac{\kappa_{[\sigma,T]}}{2}\cdot \alpha
\]
for any given $\ve\in ]0,T-t[$. By Lemma \ref{cantor-jump}, for $\mu_t$-a.e. $x\in E_{t}$ there exists $\eta_x>0$ such that
\begin{equation}\label{incc1}
  ]x-\eta_x, x+\eta_x [  ~\subset~ I^{t+\ve}(t).
\end{equation}
On the other hand, given $x \in E_{t}$ and $\eta>0$, we denote the interval
\[
  J_{x,\eta}^{\sigma/2} ~=~ ]\xi_{(t,x-\eta)}\p{\sigma/2},\xi_{(t,x+\eta)}\p{\sigma/2}[ \,,
\]
and Corollaries \ref{base-bound} and \ref{rarefaction} imply that
\begin{align*}
  \left|J_{x,\eta}^{\sigma/2} \setminus I^t(\sigma/2)\right| ~=&~ \xi_{(t,x+\eta)}(\sigma/2) -\xi_{(t,x-\eta)}(\sigma/2) - \left|J_{x,\eta}^{\sigma/2} \cap I^t(\sigma/2)\right| \nonumber \\
   ~\geq&~ \kappa_{[\sigma,T]}\cdot \left( u(t,x-\eta) - u(t,x+\eta) \right) + c_T(\sigma/2) \nu_t(]x-\eta,x+\eta[)\,.
\end{align*}
Furthermore, by \eqref{c1} and the definition of $E_t$, there exists $\eta_0 > 0$ such that
\begin{equation}
  \left|J_{x,\eta}^{\sigma/2} \setminus I^t(\sigma/2)\right|
       ~\geq~ - \frac{\kappa_{[\sigma,T]}}{2} \mu_{t}(]x-\eta,x+\eta[) \qquad\quad \forall \eta \in ]0,\eta_0] \label{JI-bound} \,.
\end{equation}
By the Besicovitch covering lemma, we can cover $\mu_{t}$-a.e. $E_{t}$ with countably many pairwise disjoint intervals $[x_j-\eta_j,x_j+\eta_j]$ where $\eta_j$ is chosen such that both \eqref{incc1}and \eqref{JI-bound} hold.
Proposition \ref{gen-char} (ii) implies that the intervals $J_{x_j,\eta_j}^{\sigma/2}$ are pairwise disjoint and by \eqref{incc1} we have that $J^{\sigma/2}_{x_j,\eta_j}$ is contained in $A_{\sigma/2}$. Therefore, it holds that
\[
    F_\sigma(t + \ve) - F_\sigma(t) ~=~ \left|  I^{t+\ve}(\sigma/2)\setminus I^t(\sigma/2)  \right| ~\geq~ \sum_j \left|J_{x_j,\eta_j}^{\sigma/2} \setminus I^t(\sigma/2)\right| \,.
\]
Applying \eqref{JI-bound} and then \eqref{cantor-present} to the above inequality yields
\begin{align*}
  F_\sigma(t + \ve) - F_\sigma(t) ~\geq~ -\frac{\kappa_{[\sigma,T]}}{2} \sum_j \mu_{t}\left([x_j-\eta_j,x_j+\eta_j]\right)
   ~\geq~ -\frac{\kappa_{[\sigma,T]}}{2} \mu_t\left(E_{t}\right)
  ~\geq~ \frac{\kappa_{[\sigma,T]}}{2} \alpha   \,,
\end{align*}
and therefore \eqref{f-jump} holds.
\v

{\bf (iii).} By the monotonicity of $F_\sigma$ and \eqref{F-bound}, $F_\sigma$ has at most countable many discontinuities on $[\sigma,T]$. Thus, for any given $\sigma \in ]0,T[$, \eqref{cantor-present}-\eqref{f-jump} imply that the set
\[
  \bigcup_{n \in \mathbb N }\sett{t \in [\sigma, T]}{ \mu_t(A_t) \leq -2^{-n} }~=~ \sett{t \in [\sigma, T]}{ \mu_t(A_t) < 0 }
\]
is at most countable and therefore,
\[
  \bigcup_{n \in \mathbb N} \sett{t \in [2^{-n}, T]}{ \mu_t(A_t) < 0 } ~=~  \mathcal{T}_{[z_1,z_2]}~~\textrm{is countable.}
\]

 {\bf Step 2.} To complete the proof, it is sufficient to show that for any given $T>0$, there exists an at most countable subset $\mathcal{T}_T$ of $[0,T]$ such that
 \begin{equation}\label{las}
u(t,\cdot)~\in~ SBV_{\loc}(\R)\qquad\forall t\in [0,T]\setminus\mathcal{T}_T\,.
\end{equation}

 For any $k\in\mathbb{Z}$, we pick a point $\bar{z}_k\in ]k,k+1[\backslash \mathcal{J}(T)$. Let $\xi_k(\cdot)$ be the unique genuine backward characteristic starting at point $(T,\bar{z}_k)$ for every $k\in\mathbb{Z}$ and  define
\[
\mathcal{A}^T_k~=~\mathcal{A}^{T}_{[\bar{z}_k,\bar{z}_{k+1}]}\bigcup \{(\xi_k(t),t): t\in [0,T]\}\quad\mathrm{and}\qquad A^T_k(t)~=~A^T_{[\bar{z}_k,\bar{z}_{k+1}]}(t)\bigcup \{\xi_k(t)\} \,.
\]
Due to the no-crossing property of two genuine backward characteristics in Proposition \ref{gen-char}, it holds that
\[
\bigcup_{k\in\mathbb{Z}}\mathcal{A}^T_k~=~ [0,T] \times \R  \qquad\mathrm{and}\qquad \bigcup_{k\in\mathbb{Z}} A^T_k(t)~=~\R\qquad\forall t\in [0,T].
\]
From Step 1, it holds that, for every $k\in\mathbb{Z}$, the set
\[
\{t\in [0,T]:\mu_t(A^T_k(t))~\neq~0\}~~\mathrm{is~countable}.
\]
Hence,
\[
\mathcal{T}_T~=~\{t\in [0,T]: \mu_t(A^T_k(t))~\neq~0~~\mathrm{for~some}~ k\in\mathbb{Z}\}~~\mathrm{is~also~countable}.
\]
and this yields (\ref{las}).
\endproof
%
%
\v

{\bf Acknowledgments.} This research by K.T.Nguyen was partially supported by a grant from
the Simons Foundation/SFARI (521811, NTK).

\end{document}